\documentclass[12pt,reqno]{amsart}
\usepackage{amsmath}
\usepackage{amssymb}
\usepackage{amstext}
\usepackage{bm}
\usepackage{a4wide}
\usepackage{enumerate}
\allowdisplaybreaks \numberwithin{equation}{section}
\usepackage{color}

\numberwithin{equation}{section}

\newtheorem{theorem}{Theorem}[section]

\newtheorem{lemma}[theorem]{Lemma}

\theoremstyle{definition}

\newtheorem{definition}[theorem]{Definition}

\theoremstyle{remark}
\newtheorem{remark}[theorem]{Remark}

\begin{document}

\title
{Steady Double Vortex Patches with Opposite Signs in a Planar Ideal Fluid}

 \author{Daomin Cao, Guodong Wang}


\begin{abstract}
In this paper we consider steady vortex flows for the incompressible Euler equations in a planar bounded domain. By solving a variational problem for the vorticity, we construct steady double vortex patches with opposite rotation directions  concentrating at a strict local minimum point of the Kirchhoff-Routh function with $k=2$. Moreover, we show that such steady solutions are in fact local maximizers of the kinetic energy among isovortical patches, which correlates stability to uniqueness.

\end{abstract}

\maketitle

\section{Introduction}
In this paper, we study steady planar vortex flows with two concentrated regions of vorticity with opposite signs in a bounded domain. By maximizing the kinetic energy subject to some constraints for the vorticity, we construct a family of steady vortex patch solutions shrinking to a given strict local minimum point of the corresponding Kirchhoff-Routh function with $k=2$. This generalizes the result of Turkington in \cite{T} for $k=1$.

Our result can be regarded as the desingularization of point vortices. According to the Kirchhoff vortex model, the motion of $k$ point vortices is described by a Hamiltonian system called the Kirchhoff equation with the Kirchhoff-Routh function as the Hamiltonian. In \cite{T3}, Turkington proved that if at initial time the vorticity is concentrated, then it remains concentrated in any finite time interval, and the location of the concentration point satisfies the Kirchhoff equation. In \cite{MPa}, the authors extended the result of \cite{T3} to two concentrated vortices with opposite signs. According to  \cite{MPa}, \cite{T3}, if a family of steady vortices shrink to a point, then it must be the critical point of the corresponding Kirchhoff-Routh function. Now conversely, given a critical point of the Kirchhoff-Routh function, can we construct a family of steady vortices shrinking to the point?

Turkington firstly considered this question. In \cite{T3}, by solving a variational problem for the vorticity, he constructed a family of vortex patch solutions shrinking to a global minimum point of the Kirchhoff-Routh function with $k=1$(in this case the Kirchhoff-Routh function is also called the Robin function). But for some domains with several global minimum points for the Kirchhoff-Routh function, which point the solutions shrink to is unknown. In \cite{CPY}, the authors gave this question a complete answer: given a non-degenerate critical point of the Kirchhoff-Routh funtion for any $k$, there exists a family of steady vortex patches shrinking to that point. The construction in \cite{CPY} was based on reduction method for the stream function, which was very different from the method of that of Turkington.

A very natural question is whether the two methods result in the same solution. The answer is yes for $k=1$ according to the local uniqueness result proved in \cite{CGPY}. As a consequence, we have an energy characterization for the single vortex patch constructed in \cite{CPY}. This idea is used in \cite{CW} to prove stability of vortex patches for $k=1$. But for $k\geq2$, uniqueness is still an open problem. Here in this paper, based on the idea of Turkington, we construct steady double vortex patches($k=2$) with different rotating directions, and show that such solutions are in fact local energy maximizers in rearrangement class. Note that we do not know whether the solutions obtained in this paper are included in \cite{CPY}, because uniqueness of double vortex patches with different rotation direction is still an open problem. In fact in Section 4 we show that uniqueness implies stability.

The study of equilibrium configurations for steady vortex flows has a long history, see for example \cite{AK,AS,Ba,BP,B1,B2,B4,CLW,CPY,EM,EM2,FT,N,Ni,SV,T,T2}. Generally speaking, there are two methods dealing with this problem: the stream-function method and the vorticity method. For steady ideal fluid, the stream function satisfies a semilinear elliptic equation. The stream-function method is to study the corresponding elliptic problem, see \cite{AS,BP,CPY,CLW,N,SV} for example. In this paper we construct steady vortex patches by maximizing the kinetic energy subject to some constraints for vorticity, which is called the vorticity method. The advantage of this method is that we can analyze the energy of the solution, which is essential to prove stability, see \cite{B3},\cite{CW} for example.

This paper is organized as follows. In Section 2, we introduce the vortex patch problem and give some notations. In Section 3, we consider a maximization problem for vorticity, by analyzing the limiting behavior we show that the maximizers are steady solutions of the Euler equations. In Section 3, we prove that the vortex patch solutions obtained in Section 2 are in fact local energy maximizers among isovortical patches(not only maximizers in the maximization problem). In Section 4, stability of these vortex patches is discussed.
\section{Vortex Patch Problem}
In this section, we give a brief introduction to the vortex patch problem.

To begin with, we introduce some notations. Let $D$ be a bounded and simply-connected planar domain with smooth boundary, and $G$ be the Green function for $-\Delta$ in $D$ with zero
Dirichlet boundary condition, which is written as
\begin{equation}\label{333}
G(x,y)=\frac{1}{2\pi}\ln \frac1{|x-y|}-h(x,y), \,\,\,x,y\in
D.
\end{equation}

Given $k$ non-zero real numbers $\kappa_1,\kappa_2,...,\kappa_k$ with $k\geq 1$ being an integer, we define the Kirchhoff-Routh function in the following way
\begin{equation}
H(x_1,\cdots,x_k)=-\sum_{i\neq j}\kappa_i\kappa_jG(x_i,x_j)+\sum_{i=1}^{k}\kappa_i^2h(x_i,x_i)
\end{equation}
where $(x_1,\cdots,x_k)\in D^{(k)}=\underbrace{D\times D\times\cdots\times D}_{k}$ such that $x_i\neq x_j$ for $i\neq j$.
In this paper we focus on the case that $k=2$, $\kappa_1>0,\kappa_2<0$, and the Kirchhoff-Routh function can be written as
\begin{equation}\label{211}
H(x_1,x_2)=-2\kappa_1\kappa_2G(x_1,x_2)+\sum_{i=1}^{2}\kappa_i^2h(x_i,x_i),
\end{equation}
where $(x_1,x_2)\in D^{(2)}$ and $x_1\neq x_2$.

The motion of steady ideal fluid in $D$ with impermeability boundary condition is described by the following Euler system:
\begin{equation}\label{1}
\begin{cases}

(\mathbf{v}\cdot\nabla)\mathbf{v}=-\nabla P \,\,\,\,\,\,\,\,\,\,\,\,\,\text{in $D$},\\
\nabla\cdot\mathbf{v}=0 \,\,\,\,\,\,\,\,\,\,\,\,\,\,\,\,\,\,\,\,\,\,\,\,\,\,\,\,\,\,\,\,\,\text{in $D$},\\
\mathbf{v}\cdot \vec{n}=0 \,\,\,\,\,\,\,\,\,\,\,\,\,\,\,\,\,\,\,\,\,\,\,\,\,\,\,\,\,\,\,\,\,\,\text{on $\partial D$ },
\end{cases}
\end{equation}
where $\mathbf{v}=(v_1,v_2)$ is the velocity field, $P$ is the pressure and $\vec{n}$ is the outward unit normal of $\partial D$. Here we assume that the fluid has unit density.

By defining the vorticity $\omega \triangleq curl\mathbf{v} =\partial_1 v_2-\partial_2v_1$ and using the identity $\frac{1}{2}\nabla|\mathbf{v}|^2=(\mathbf{v}\cdot\nabla)\mathbf{v}+J\mathbf{v}\omega$,
 the first equation of $\eqref{1}$ can be written as
\begin{equation}\label{3}
 \nabla(\frac{1}{2}|\mathbf{v}|^2+P)-J\mathbf{v}\omega=0,
\end{equation}
where $J(v_1,v_2)=(v_2,-v_1)$ denotes clockwise rotation through $\frac{\pi}{2}$. Taking the curl in $\eqref{3}$ gives

\begin{equation}\label{499}
 \mathbf{v}\cdot\nabla\omega=0.
\end{equation}

By divergence-free condition, $\mathbf{v}$ can be written as $\mathbf{v}=J\nabla\psi$ for some function $\psi$ called the stream function. It is easy to verify that $\psi$ satisfies the following equation:
\begin{equation}
 \begin{cases}
-\Delta \psi=\omega\text{ \quad \,\quad in $D$,}\\
\psi= \text{constant} \text{\quad on $\partial D$.}
\end{cases}
\end{equation}
Without loss of generality, we assume that $\psi$ vanishes on $\partial D$ by adding a constant, so formally $\psi(x)=\int_DG(x,y)\omega(y)dy$.
By introducing the notation
 $\partial(\psi,\omega)\triangleq\partial_1\psi\partial_2\omega-\partial_2\psi\partial_1\omega$, $\eqref{499}$ can be written as
\begin{equation}\label{599}
\partial(\omega,\psi)=0.
\end{equation}

Integrating by parts gives the following weak form of $\eqref{599}$:
 \begin{equation}\label{997*}
  \int_D\omega\partial(\xi,\psi)dx=0
  \end{equation}
  for all $\xi\in C_0^{\infty}(D)$. Note that for $\omega\in L^{\infty}(D)$, $\psi\in W^{2,p}(D)$ for any $1\leq p<+\infty$ by $L^p$ estimate thus $\psi\in C^{1,\alpha}(D)$ for some $0<\alpha<1$ by Sobolev embedding, so the definition of weak solution makes sense for all $\omega\in L^{\infty}(D)$.

 In the sequel, $I_{A}$ will denote the characteristic function of some measurable set $A$, i.e., $I_A\equiv1$ in $A$ and $I_A\equiv0$ elsewhere. If $\omega$ has the form $\omega=\lambda I_{A}$ and satisfies \eqref{997*}, we call it a steady vortex patch(see section 4 for the definition of non-steady vortex patch). Here $\lambda\in \mathbb R$ represents the vorticity strength. When $\lambda<0$, the fluid with vorticity $\omega=\lambda I_{A}$ rotates clockwisely in $A$(no rotation outside $A$), while the fluid rotates anti-clockwisely when $\lambda>0$ in $A$(no rotation outside $A$). If $\omega=\lambda_1 I_{A_1}+\lambda_2 I_{A_2}$, $dist(A_1,A_2)>0$ and $\lambda_1\lambda_2<0$, then the fluid with vorticity $\omega=\lambda I_{A}$ is said to have rotation with opposite directions in two separated regions $A_1$ and $A_2$(no rotation outside $A_1\cup A_2$), which is the case we deal with in this paper.

For any given positive integer $k$, the vortex patch problem is that, for any critical point $(x_1,\cdots,x_k)$ of the Kirchhoff-Routh function,
 \begin{enumerate}[(1)]
\item can we construct a family of steady vortex patches, the vorticity set(i.e., the set $\{\omega\neq 0\}$) of which consisting of $k$ simply-connected sets shrinking to $x_i$ separately?

\item is there local uniqueness for these steady vortex patches?

\item are these steady vortex patches stable?
\end{enumerate}
As has been mentioned in Section 1, the answer to $(1)$ is yes for any $k$, and the answer to $(2),(3)$ is yes for $k=1$. In the present paper we deal with the case $k=2$ with $\kappa_1\kappa_2<0$, especially we prove that in this case $(2)$ implies $(3)$.
\section{Construction of Solutions}
In this section we construct steady double vortex patches with opposite rotation direction by solving a variational problem for the vorticity. The procedure is similar to \cite{T} where the existence of a single vortex patch was considered.

In the sequel, we always assume that $k=2$ and $\kappa_1>0,\kappa_2<0$ for the Kirchhoff-Routh function.

\subsection{Variational problem} We consider the following maximization problem.

Assume that $(\bar{x}_1,\bar{x}_2)\in D^{(2)}$ with $\bar{x}_1\neq\bar{x}_2$ is a strict local minimum point of $H$. By choosing $\delta>0$ sufficiently small we assume that $(\bar{x}_1,\bar{x}_2)$ is the unique minimum point of $H$ in $\overline{B_\delta(\bar{x}_1)}\times \overline{B_\delta(\bar{x}_2)}$ with $B_\delta(\bar{x}_i)\subset\subset D$ for $ i=1,2$, and $\overline{B_\delta(\bar{x}_1)}\cap \overline{B_\delta(\bar{x}_2)}=\varnothing$. For convenience we denote $B_\delta(\bar{x}_i)$ by $B_i$ for $i=1,2$. Define the vorticity class as follows:
\begin{equation}\label{212}
K_\lambda=\left\{\omega=\omega_1+\omega_2\,\mid\,\,\omega_i\in L^\infty(D),\,\,0\leq sgn(\kappa_i)\omega_i\leq\lambda, \int_D\omega_i=\kappa_i, supp(\omega_i)\subset B_i,\,i=1,2\right\}.
\end{equation}
Here and in the sequel $sgn(\kappa)$ is the sign of any real number $\kappa$ and $supp(f)$ is the support of any function $f$. Note that $K_\lambda$ is not empty if $\lambda>0$ is large enough. The kinetic energy of the vortex flow with vorticity $\omega$ is defined by
\begin{equation}\label{213}
E(\omega)=\frac{1}{2}\int_D\int_DG(x,y)\omega(x)\omega(y)dxdy.
\end{equation}
In the sequel we will maximize $E$ on $K_\lambda$ and show that the maximizer satisfies $\eqref{997*}$ provided $\lambda$ is sufficiently large.
\subsection{Existence of maximizer}
The existence of maximizer of $E$ on $K_\lambda$ can be easily proved by using standard maximization techniques.
\begin{theorem}
There exists $\omega^\lambda\in K_\lambda$ such that $E(\omega^\lambda)=\sup_{\omega\in K_\lambda}E(\omega)$.
\end{theorem}
\begin{proof}
Firstly since $G(x,y)\in L^{1}(D\times D)$, there exists a constant $C>0$ independent of $\lambda$ such that for any $\omega\in K_\lambda$
\[E(\omega)=\frac{1}{2}\int_D\int_DG(x,y)\omega(x)\omega(y)dxdy\leq \frac{1}{2}\lambda^2\int_D\int_D|G(x,y)|dxdy\leq C\lambda^2,
\]
 which means $\sup_{\omega\in K_\lambda}E(\omega)<+\infty$. Let $\{\omega_n\}\subset K_\lambda$ be a maximizing sequence such that $\lim_{n\rightarrow+\infty}E(\omega_n)= \sup_{\omega\in K_\lambda}E(\omega)$. Since $K_\lambda$ is bounded in $\L^\infty(D)$, $K_\lambda$ is sequentially compact in the weak star topology in $L^\infty(D)$. Without loss of generality we assume that $\omega_n\rightarrow \omega^\lambda$ weakly star in $L^\infty(D)$ for some $\omega^\lambda\in L^\infty(D)$ as $n\rightarrow +\infty$.

We claim that $\omega^\lambda\in K_\lambda$. In fact, $\omega_n\rightarrow \omega^\lambda$ weakly star in $L^\infty(D)$ means
\[
\lim_{n\rightarrow +\infty}\int_D\omega_n\phi=\int_D\omega^\lambda\phi,\,\, \phi\in L^1(D).
\]
Then for any $\phi\in C_0^\infty(D\setminus \cup_{i=1}^2 B_i)$,
\[
\lim_{n\rightarrow +\infty}\int_D\omega_n\phi=\int_D\omega^\lambda\phi=0,
\]
which means $\omega^\lambda\equiv0$ a.e. in $D\setminus \cup_{i=1}^2B_i$. Thus $supp(\omega^\lambda)\subset\cup_{i=1}^2 B_i$. Define $\omega_i^\lambda=\omega^\lambda I_{B_i}$, then we have $\omega^\lambda=\sum_{i=1}^{2}\omega^\lambda_i$. Now by choosing $\phi=I_{B_i}$, we have
\[
\int_D\omega_i^\lambda=\int_D\omega^\lambda\phi=\lim_{n\rightarrow +\infty}\int_D\omega_n\phi=\lim_{n\rightarrow +\infty}\int_{B_i}\omega_n=\kappa_i.
\]
Now we prove that $0\leq sgn(\kappa_i)\omega_i^\lambda\leq\lambda$ for any fixed $i=1,2$. Without loss of generality we assume $\kappa_i>0$. We prove $\omega_i^\lambda\leq\lambda$ first. Suppose that $|\{\omega^\lambda_i>\lambda\}|>0$, then there exist $\varepsilon_0,\varepsilon_1>0$ such that $|\{\omega^\lambda_i>\lambda+\varepsilon_0\}|>\varepsilon_1$. Denote $B_i^*=\{\omega^\lambda_i>\lambda+\varepsilon_0\}\subset B_\delta(\bar{x}_i)$, then for $\phi=I_{B_i^*}$ we have
\[\int_{B_i^*}(\omega^\lambda_i-\omega_n)\geq\varepsilon_0|B_i^*|\geq\varepsilon_0\varepsilon_1.\]
On the other hand, by weak star convergence
\[\lim_{n\rightarrow +\infty}\int_{B_i^*}(\omega^\lambda_i-\omega_n)=\lim_{n\rightarrow +\infty}\int_D(\omega^\lambda_i-\omega_n)\phi=0,\]
which is a contradiction, so $\omega^\lambda_i\leq \lambda$. Similar argument gives $\omega^\lambda_i\geq 0$. Thus the claim is proved.

Finally by weak star convergence we have
\[\lim_{n\rightarrow +\infty}\frac{1}{2}\int_D\int_DG(x,y)\omega_n(x)\omega_n(y)dxdy=\frac{1}{2}\int_D\int_DG(x,y)\omega^\lambda(x)\omega^\lambda(y)dxdy\]
 which means $E(\omega^\lambda)=\lim_{n\rightarrow +\infty} E(\omega_n)=\sup_{\omega\in K_\lambda}E(\omega)$.
\end{proof}

\subsection{Profile of $\omega^\lambda$} Now we show that $\omega^\lambda$ has a special form.
\begin{lemma}
$\omega^\lambda$ has the form $\omega^\lambda=\sum_{i=1}^2\omega^\lambda_i$, where $\omega^\lambda_i=sgn(\kappa_i)\lambda I_{\Omega_i}$,and $\Omega_i$ is defined by
\[\Omega_i={\{sgn(\kappa_i)\psi^\lambda>\mu^\lambda_i\}\cap B_i}\]
for real numbers $\mu^\lambda_i$ depending on $\omega^\lambda$, $i=1,2$. Here $\psi^\lambda(x)=\int_DG(x,y)\omega^\lambda(y)dy$.
\end{lemma}
\begin{proof}
Without loss of generality we only prove the case $i=1$, similar argument applies to the case $i=2$. It suffices to show that $\omega^\lambda_1=\lambda I_{\{\psi^\lambda>\mu^\lambda_1\}\cap B_1}$ for some real number $\mu^\lambda_1$.

Define a family of test functions $\omega_{(s)}(x)=\omega^\lambda+s[z_0(x)-z_1(x)]$, $s>0$, where $z_0,z_1$ satisfies
\begin{equation}\label{777}
\begin{cases}
z_0,z_1\in L^\infty(D),\,\, z_0,z_1\geq 0,\,\, \int_Dz_0=\int_D z_1,
 \\supp(z_0),supp(z_1)\subset B_1,
 \\z_0=0\text{\,\,\,\,\,\,} in\text{\,\,} D\verb|\|\{\omega^\lambda_1\leq\lambda-\varepsilon\},
 \\z_1=0\text{\,\,\,\,\,\,} in\text{\,\,} D\verb|\|\{\omega^\lambda_1\geq\varepsilon\},
\end{cases}
\end{equation}
where $\varepsilon$ is any small positive number.

Note that for fixed $z_0,z_1$ and $\varepsilon$, $\omega_{(s)}\in K_\lambda$ provided $s$ is sufficiently small. Since $\omega_{(0)}=\omega^\lambda$ is a maximizer, we have
\begin{equation}
0\geq\frac{dE(\omega_{(s)})}{ds}|_{s=0^+}=\int_Dz_0\psi^\lambda-\int_Dz_1\psi^\lambda,
\end{equation}
where $\psi^\lambda(x)=\int_DG(x,y)\omega^\lambda(y)dy.$
That is $\int_Dz_0\psi^\lambda\leq\int_Dz_1\psi^\lambda$ for any $z_0,z_1$ satisfying \eqref{777}, which gives
\begin{equation}\label{778}
\sup_{\{\omega^\lambda<\lambda\}\cap B_1}\psi^\lambda\leq\inf_{\{\omega^\lambda>0\}\cap B_1}\psi^\lambda.\end{equation}
Since $\overline{B_1}$ is connected and $\overline{\{\omega^\lambda<\lambda\}\cap B_1}\cup\overline{\{\omega^\lambda>0\}\cap B_1}=\overline{B_1}$, we have $\overline{\{\omega^\lambda<\lambda\}\cap B_1}\cap\overline{\{\omega^\lambda>0\}\cap B_1}\neq\varnothing$, then by the continuity of $\psi^\lambda$($\psi^\lambda$ satisfies $-\Delta\psi^\lambda=\omega^\lambda\in L^\infty(D)$, by the regularity theory of elliptic equations $\psi^\lambda\in C^{1,\alpha}(D)$ for any $0<\alpha<1$),
\begin{equation}\label{779}
\sup_{\{\omega^\lambda<\lambda\}\cap B_1}\psi^\lambda\geq\inf_{\{\omega^\lambda>0\}\cap B_1}\psi^\lambda.\end{equation}
\eqref{778} and \eqref{779} together give
\begin{equation}
\sup_{\{\omega^\lambda<\lambda\}\cap B_1}\psi^\lambda=\inf_{\{\omega^\lambda>0\}\cap B_1}\psi^\lambda,
\end{equation}
from we can define $\mu^\lambda_1\triangleq\sup_{\{\omega^\lambda<\lambda\}\cap B_1}\psi^\lambda=\inf_{\{\omega^\lambda>0\}\cap B_1}\psi^\lambda$. Then obviously
\begin{equation}
\begin{cases}
\omega^\lambda=0\text{\,\,\,\,\,\,$a.e.$\,} in\text{\,\,}\{\psi^\lambda_1<\mu^\lambda_1\}\cap B_1,
 \\ \omega^\lambda=\lambda\text{\,\,\,\,\,\,$a.e.$\,} in\text{\,\,}\{\psi^\lambda_1>\mu^\lambda_1\}\cap B_1,
\end{cases}
\end{equation}
It remains to show that $\omega^\lambda=0$ a.e. on $\{\psi^\lambda_1=\mu^\lambda_1\}\cap B_1$. In fact, on $\{\psi^\lambda_1=\mu^\lambda_1\}\cap B_1$, $\psi^\lambda$ is constant, so we have $\nabla\psi^\lambda=0$,  then $\omega^\lambda=-\Delta \psi^\lambda =0$.
Altogether, $\omega^\lambda_1$ satisfies
\begin{equation}
\begin{cases}
\omega^\lambda_1=0\text{\,\,\,\,\,\,$a.e.$\,} in\text{\,\,}\{\psi^\lambda_1\leq\mu^\lambda_1\}\cap B_1,
 \\ \omega^\lambda_1=\lambda\text{\,\,\,\,\,\,$a.e.$\,} in\text{\,\,}\{\psi^\lambda_1>\mu^\lambda_1\}\cap B_1,
\end{cases}
\end{equation}
which completes the proof.

\end{proof}
\begin{remark}\label{670}
It is easy to verify that $\mu^\lambda_1>0,\mu^\lambda_2>0$ if $\delta$ is sufficiently small and $\lambda$ is sufficiently large. In fact, for $x\in \partial B_\delta(\bar{x}_1)$,
\begin{equation}\label{334}
\begin{split}
\psi^\lambda(x)&=\int_DG(x,y)\omega^\lambda(y)dy\\
=&\int_{B_\delta(\bar{x}_1)}G(x,y)\omega^\lambda_1(y)dy+\int_{B_\delta(\bar{x}_2)}G(x,y)\omega^\lambda_2(y)dy\\
=&\int_{B_\delta(\bar{x}_1)}-\frac{1}{2\pi}\ln|x-y|\omega^\lambda_1(y)dy-\int_{B_\delta(\bar{x}_1)}h(x,y)\omega^\lambda_1(y)dy+\int_{B_\delta(\bar{x}_2)}G(x,y)\omega^\lambda_2(y)dy.\\
\end{split}
\end{equation}
Since $|x-y|\leq \delta$ for $y\in B_\delta(\bar{x}_1)$, the first term in \eqref{334}
\begin{equation}
\int_{B_\delta(\bar{x}_1)}-\frac{1}{2\pi}\ln|x-y|\omega^\lambda_1(y)dy\geq-\frac{\kappa_1}{2\pi}\ln\delta.
\end{equation}
As for the second term and the third term, we notice that $dist(B_\delta(\bar{x}_1),\partial D)>0$ and $dist(B_\delta(\bar{x}_1),B_\delta(\bar{x}_2))>0$, so there exists some $C>0$ independent of $\delta$ and $\lambda$ such that
\begin{equation}
-\int_{B_\delta(\bar{x}_1)}h(x,y)\omega^\lambda_1(y)dy+\int_{B_\delta(\bar{x}_2)}G(x,y)\omega^\lambda_2(y)dy\geq-C.
\end{equation}
So we have
\begin{equation}
\psi^\lambda(x)\geq-\frac{\kappa_1}{2\pi}\ln\delta-C
\end{equation}
 for any $x\in \partial B_\delta(\bar{x}_1)$.  By choosing $\delta$ sufficiently small(necessarily $\lambda$ sufficiently large to ensure $K_\lambda\neq \varnothing$), we have $\psi^\lambda>0$ on $\partial B_\delta(\bar{x}_1)$, thus $\psi^\lambda>0$ in $B_\delta(\bar{x}_1)$ by maximum principle. On the other hand, we notice that $|\Omega_1|\rightarrow0$ as $\lambda\rightarrow+\infty$, so $\partial\Omega_1\cap B_1\neq \varnothing$ for sufficiently large $\lambda$. Thus by choosing $x^*\in\partial\Omega_1\cap B_1$ we have $\mu^\lambda_1=\psi^\lambda(x^*)>0$.  Similarly $\mu^\lambda_2>0$ if $\delta$ is sufficiently small.

 In the sequel, we choose $\delta$ small enough such that $\mu^\lambda_1>0,\mu^\lambda_2>0$.
\end{remark}

Now we have proved the existence of $\omega^\lambda$ as the maximizer of the variational problem, however, $\omega^\lambda$ does not necessarily satisfy $\eqref{997*}$. To show that $\omega^\lambda$ is a steady vortex patch, we need to prove that the support of $\omega^\lambda_i$ is away from $\partial B_i$. This is achieved by analyzing the asymptotic behavior of $\omega^\lambda$ as $\lambda\rightarrow +\infty$ in the next subsection.

\subsection{Asymptotic estimates}
In this subsection, we prove that the support of $\omega^\lambda_i$ shrinks to $\bar{x}_i$, that is, the diameter of $supp(\omega^\lambda_i)$ goes to zero and the center of $supp(\omega^\lambda_i)$ converges to $\bar{x}_i$ as $\lambda\rightarrow +\infty$.

First we estimate the lower bound of the total energy.
\begin{lemma}\label{566}
$E(\omega^\lambda)\geq-\frac{1}{4\pi}\sum_{i=1}^{2}\kappa_i^2\ln\varepsilon_i-C$, where $\varepsilon_i$ satisfies $\lambda|B_{\varepsilon_i}(\bar{x}_i)|=|\kappa_i|$ $($or equivalently $\varepsilon_i=\sqrt{\frac{|\kappa_i|}{\lambda\pi}})$.
\end{lemma}
\begin{proof}
The result is proved by choosing a proper test function. Indeed, take the test function $\bar{\omega}=\sum_{i=1}^{2}\bar{\omega}_i$, where $\bar{\omega}_i=sgn(\kappa_i)\lambda I_{B_{\varepsilon_i}(\bar{x}_i)}$, then it is easy to check that $\bar{\omega}\in K_\lambda$, so we have the lower bound for $E(\omega^\lambda)$:
\[E(\omega^\lambda)\geq E(\bar{\omega}).\]
Now we calculate $E(\bar{\omega})$,
\begin{equation}\label{300}
\begin{split}
  E(\bar{\omega})=&\frac{1}{2}\int_D\int_D G(x,y)\bar{\omega}(x)\bar{\omega}(y)dxdy \\
 =&\frac{1}{2}\sum_{1\leq i,j\leq 2}\int_DG(x,y)\bar{\omega}_i(x)\bar{\omega}_j(y)dxdy\\
 =&\frac{1}{2}\sum_{i=1}^{2}\int_D\int_DG(x,y)\bar{\omega}_i(x)\bar{\omega}_i(y)dxdy+\int_D\int_DG(x,y)\bar{\omega}_1(x)\bar{\omega}_2(y)dxdy. \\
\end{split}
\end{equation}
For fixed $i$, $\int_D\int_DG(x,y)\bar{\omega}_i(x)\bar{\omega}_i(y)dxdy$ has the following estimate:
\begin{equation}\label{301}
\begin{split}
  \int_D\int_DG(x,y)\bar{\omega}_i(x)\bar{\omega}_i(y)dxdy =& \lambda^2\int_D\int_DG(x,y)I_{B_{{\varepsilon}_i}}(\bar{x}_i)(x)I_{B_{{\varepsilon}_i}}(\bar{x}_i)(y)dxdy \\
=&\lambda^2\int_{B_{{\varepsilon}_i}(\bar{x}_i)}\int_{B_{{\varepsilon}_i}(\bar{x}_i)}G(x,y)dxdy\\
=&\frac{\lambda^2}{2\pi}\int_{B_{{\varepsilon}_i}(\bar{x}_i)}\int_{B_{{\varepsilon}_i}(\bar{x}_i)}\ln\frac{1}{|x-y|}dxdy-
\lambda^2\int_{B_{{\varepsilon}_i}(\bar{x}_i)}\int_{B_{{\varepsilon}_i}(\bar{x}_i)}h(x,y)dxdy\\
=&\frac{\lambda^2}{2\pi}\int_{B_{{\varepsilon}_i}(\bar{x}_i)}\int_{B_{{\varepsilon}_i}(\bar{x}_i)}\ln\frac{1}{|x-y|}dxdy-\kappa_i^2h(\bar{x}_i,\bar{x}_i)+o(1)\\
\geq&\frac{\lambda^2}{2\pi}\int_{B_{{\varepsilon}_i}(\bar{x}_i)}\int_{B_{{\varepsilon}_i}(\bar{x}_i)}\ln\frac{1}{2\varepsilon_i}dxdy-\kappa_i^2h(\bar{x}_i,\bar{x}_i)+o(1)\\
\geq &-\frac{\kappa_i^2}{2\pi}\ln\varepsilon_i-C.
\end{split}
\end{equation}
Here as before $C$ denotes various constants independent of $\lambda$, and $o(1)$ denotes various quantities tending to 0 as $\lambda\rightarrow+\infty$.

On the other hand, the integral $\int_D\int_DG(x,y)\bar{\omega}_1(x)\bar{\omega}_2(y)dxdy$ is uniformly bounded:
\begin{equation}\label{302}
\begin{split}
&\left|\int_D\int_DG(x,y)\bar{\omega}_1(x)\bar{\omega}_2(y)dxdy\right|\\=&\left|\lambda^2\int_{B_{{\varepsilon}_1}(\bar{x}_1)}\int_{B_{{\varepsilon}_2}(\bar{x}_2)}G(x,y)dxdy\right|\\
\leq&\left|\frac{\lambda^2}{2\pi}\int_{B_{{\varepsilon}_i}(\bar{x}_i)}\int_{B_{{\varepsilon}_j}(\bar{x}_j)}\ln\frac{1}{|x-y|}dxdy\right|+
\left|\lambda^2\int_{B_{{\varepsilon}_1}(\bar{x}_1)}\int_{B_{{\varepsilon}_2}(\bar{x}_2)}h(x,y)dxdy\right|\\
\leq&\lambda^2\int_{B_{{\varepsilon}_1}(\bar{x}_1)}\int_{B_{{\varepsilon}_2}(\bar{x}_2)}Cdxdy\\
\leq&C,
\end{split}
\end{equation}
where we use the fact that $dist(B_{\varepsilon_1}(\bar{x}_1),B_{\varepsilon_2}(\bar{x}_2)$ has a positive lower bound independent of $\lambda$.

Combining $\eqref{300}$, $\eqref{301}$ and $\eqref{302}$ we get the desired result.

\end{proof}

Now we divide the total energy $E(\omega^\lambda)$ into two parts:
\begin{equation}\label{303}
\begin{split}
E(\omega^\lambda)=&\frac{1}{2}\int_D\psi^\lambda\omega^\lambda
=\frac{1}{2}\sum_{i=1}^2\int_D\psi^\lambda\omega^\lambda_i\\
=&\frac{1}{2}\sum_{i=1}^2\int_D(\psi^\lambda-sgn(\kappa_i)\mu_i^\lambda)\omega^\lambda_i
+\frac{1}{2}\sum_{i=1}^2\int_D\mu_i^\lambda sgn(\kappa_i)\omega^\lambda_i\\
=&\frac{1}{2}\sum_{i=1}^2\int_D(\psi^\lambda-sgn(\kappa_i)\mu_i^\lambda)\omega^\lambda_i+\frac{1}{2}\sum_{i=1}^2\mu_i^\lambda|\kappa_i|\\
\end{split}
\end{equation}
The first term in the above formula is called the energy of the vortex core, written as
\begin{equation}
T(\omega^\lambda)\triangleq \frac{1}{2}\sum_{i=1}^2\int_D(\psi^\lambda-sgn(\kappa_i)\mu_i^\lambda)\omega^\lambda_i.
\end{equation}
We have the following estimate for $T(\omega^\lambda)$.
\begin{lemma}\label{422}
$T(\omega^\lambda)\leq C$.
\end{lemma}
\begin{proof}
Denote $\zeta_i=\psi^\lambda-sgn(\kappa_i)\mu_i^\lambda, T_1=\int_D\zeta_1\omega^\lambda_1, T_2=\int_D\zeta_2\omega^\lambda_2$, then $T(\omega^\lambda)=\frac{1}{2}(T_1+T_2)$. It suffices to prove that $T_1, T_2$ are both uniformly bounded from above. We only consider $T_1$, and the same argument applies to $T_2$.

On the one hand,
\begin{equation}\label{677}
\begin{split}
T_1=&\int_D\zeta_1\omega^\lambda_1
  =\lambda\int_{\Omega_1}\zeta_1\\
  \leq&\lambda|\Omega_1|^{\frac{1}{2}}(\int_{\Omega_1}\zeta^2_1)^\frac{1}{2}\\
   \leq&\lambda|\Omega_1|^{\frac{1}{2}}(\int_{B_1}\zeta^{+2}_1)^\frac{1}{2}\\
  \leq&C\lambda|\Omega_1|^{\frac{1}{2}}(\int_{B_1}\zeta^+_1+\int_{B_1}|\nabla\zeta^+_1|)\\
  =&C\lambda|\Omega_1|^{\frac{1}{2}}\int_{B_1}\zeta^+_1+   C\lambda|\Omega_1|^{\frac{1}{2}}\int_{B_1}|\nabla\zeta^+_1|\\
  =&C|\Omega_1|^{\frac{1}{2}}T_1  +   C\lambda|\Omega_1|^{\frac{1}{2}}\int_{B_1}|\nabla\zeta^+_1|\\
  \leq&C|\Omega_1|^{\frac{1}{2}}T_1  +   C(\int_{\Omega_1}|\nabla\zeta^+_1|^2)^\frac{1}{2}\\
  =&o(1)T_1  +   C(\int_{\Omega_1}|\nabla\zeta^+_1|^2)^\frac{1}{2}
\end{split}
\end{equation}
as $\lambda\rightarrow +\infty$. Here we use the Sobolev embedding $W^{1,1}(B_1)\hookrightarrow L^2(B_1)$.
From \eqref{677} we have
\begin{equation}\label{697}
T_1\leq  C(\int_{\Omega_1}|\nabla\zeta^+_1|^2)^\frac{1}{2}.
\end{equation}
One the other hand,
\begin{equation}\label{986}
\begin{split}
T_1
=&\int_D\zeta_1\omega^\lambda_1
=\int_D\zeta^+_1\omega^\lambda_1
\geq\int_D\zeta_1^+\omega^\lambda
=\int_D\zeta_1^+(-\Delta \zeta_1)
=\int_{\{\zeta_1>0\}}|\nabla\zeta_1|^2
\geq\int_{\Omega_1}|\nabla\zeta^+_1|^2.
\end{split}
\end{equation}
Here $\omega^\lambda_1\geq\omega^\lambda$ in $D$ since $\kappa_2<0$, and $\int_D\zeta_1^+(-\Delta \zeta_1)=\int_{\{\zeta_1>0\}}|\nabla\zeta_1|^2$ since $\{\zeta_1>0\}\subset\subset D$ by Remark \ref{670}.

\eqref{697} and \eqref{986} together give $T_1\leq C$, which is the desired result.

\end{proof}
By combining Lemma \ref{566}, \eqref{303} and Lemma \ref{422}, we obtain the following lower bound of $\sum_{i=1}^2|\kappa_i|\mu^\lambda_i$.
\begin{lemma}
$\sum_{i=1}^2|\kappa_i|\mu^\lambda_i\geq-\frac{1}{2\pi}\sum_{i=1}^{2}\kappa_i^2\ln\varepsilon_i-C.$
\end{lemma}

Now we are able to estimate the diameter of the support of $\omega^\lambda_i$.
\begin{lemma}\label{351}
There exists $R>1$ such that $diam(supp(\omega^\lambda_i))<R\varepsilon_i$ for $i=1,2$, provided $\lambda$ is sufficiently large.
\end{lemma}
\begin{proof}
For any $x\in supp(\omega^\lambda_i)$ there holds $sgn(\kappa_i)\psi^\lambda(x)>\mu^\lambda_i$, so from the definition of $\psi^\lambda$
\[\frac{1}{2\pi}\int_D\ln\frac{1}{|x-y|}sgn(\kappa_i)\omega^\lambda_idy>\mu^\lambda_i-C.\]
By the estimate of $\mu^\lambda_i$ we have
\[\sum_{i=1}^2\frac{|\kappa_i|}{2\pi}\int_D\ln\frac{1}{|x-y|}|\omega^\lambda_i(y)|dy\geq -\frac{1}{2\pi}\sum_{i=1}^2\kappa_i^2\ln\varepsilon_i-C,\]
which implies
\begin{equation}\label{1000}
\sum_{i=1}^2\frac{|\kappa_i|}{2\pi}\int_D\ln\frac{\varepsilon_i}{|x-y|}|\omega^\lambda_i(y)|dy\geq-C.
\end{equation}
On the other hand, for any $r>0$ to be determined, we have the following inequality
\begin{equation}\label{1001}
\frac{1}{2\pi}\int_{B_{r\varepsilon_i}(x)}\ln\frac{\varepsilon_i}{|x-y|}|\omega^\lambda_i(y)|dy\leq C.
\end{equation}
From \eqref{1000},\eqref{1001} we have
\[\sum_{i=1}^2\frac{|\kappa_i|}{2\pi}\int_{B_i\verb|\|B_{r\varepsilon_i}(x)}\ln\frac{\varepsilon_i}{|x-y|}|\omega^\lambda_i(y)|dy\geq-C,\]
thus
\[\sum_{i=1}^2\int_{B_i\verb|\|B_{r\varepsilon_i}(x)}|\omega^\lambda_i(y)|dy<C\frac{1}{\ln r}.\]
Then for fixed $i$,
\[\int_{B_i\verb|\|B_{r\varepsilon_i}(x)}|\omega^\lambda_i(y)|dy<C\frac{1}{\ln r},\]
so by choosing $r$ large enough we have for any $x\in supp(\omega^\lambda_i)$
\begin{equation}\label{367}
\int_{B_{r\varepsilon_i}(x)}|\omega^\lambda_i|>\frac{1}{2}|\kappa_i|.
\end{equation}
We claim that $diam(supp(\omega^\lambda_i))<2r\varepsilon_i$. In fact, if $diam(supp(\omega^\lambda_i))\geq2r\varepsilon_i$, then we can choose $x,y\in supp(\omega^\lambda_i)$ and $|x-y|\geq 2r\varepsilon_i$, then by \eqref{367}
\[\int_D|\omega^\lambda_i|\geq \int_{B_{r\varepsilon_i}(x)}|\omega^\lambda_i|+\int_{B_{r\varepsilon_i}(y)}|\omega^\lambda_i|>|\kappa_i|,\]
which is a contradiction.
The result then follows by taking $R=2r$.
\end{proof}

Finally, by comparing energy we can prove that the support of $\omega^\lambda_i$ shrinks to $\bar{x}_i$.
\begin{lemma}\label{352}
 $\frac{1}{\kappa_i}\int_Dx\omega^\lambda_i(x)dx\rightarrow \bar{x}_i$ as $\lambda \rightarrow +\infty$.
\end{lemma}
\begin{proof}
Denote $x^\lambda_i= \frac{1}{\kappa_i}\int_Dx\omega^\lambda_i(x)dx$, then $x^\lambda_i\in \overline{B_i}$. For any sequence $\{x^{\lambda_j}_i\},\lambda_j\rightarrow +\infty$, there exists a subsequence $\{x^{\lambda_{j_k}}_i\}$ such that
$x^{\lambda_{j_k}}_i\rightarrow x_i^*\in \overline{B_i}$. For simplicity, we still denote the subsequence by $\{x^{\lambda_k}_i\}$. It suffices to show that $x_i^*=\bar{x}_i$.

As in Lemma \ref{566} we define test function $\bar{\omega}^\lambda=\sum_{i=1}^{2}\bar{\omega}^\lambda_i$, where $\bar{\omega}^\lambda_i=sgn({\kappa_i})\lambda I_{B_{\varepsilon_i}(\bar{x}_i)}$. It is easy to check that $\bar{\omega}^\lambda\in K_\lambda(D)$, so we have $ E(\bar{\omega}^\lambda)\leq E(\omega^\lambda)$, which means
\[\begin{split}
\int_D\int_D G(x,y)\bar{\omega}^{\lambda_k}(x)\bar{\omega}^{\lambda_k}(y)dxdy
\leq \int_D\int_D G(x,y)\omega^{\lambda_k}(x)\omega^{\lambda_k}(y)dxdy,
\end{split}\]
or equivalently
\begin{equation}\label{79}
\sum_{1\leq i,j\leq2}\int_D\int_D G(x,y)\bar{\omega}_i^{\lambda_k}(x)\bar{\omega}_j^{\lambda_k}(y)dxdy
\leq \sum_{1\leq i,j\leq2}\int_D\int_D G(x,y)\omega_i^{\lambda_k}(x)\omega_j^{\lambda_k}(y)dxdy.
\end{equation}
On the other hand, Riesz rearrangement inequality(see \cite{LL}, 3.7) gives
\begin{equation}\label{71}
\sum_{i=1}^{2}\int_D\int_D -\frac{1}{2\pi}\ln\frac{1}{|x-y|}\bar{\omega}_i^{\lambda_k}(x)\bar{\omega}_i^{\lambda_k}(y)dxdy
\geq \sum_{i=1}^{2}\int_D\int_D -\frac{1}{2\pi}\ln\frac{1}{|x-y|}\omega_i^{\lambda_k}(x)\omega_i^{\lambda_k}(y)dxdy.
\end{equation}
Adding \eqref{71} to \eqref{79} gives
\begin{equation}\label{787}
\begin{split}
\sum_{i\neq j}\int_D\int_DG(x,y)\bar{\omega}_i^{\lambda_k}(x)\bar{\omega}_j^{\lambda_k}(y)dxdy-\sum_{i=1}^{2}\int_D\int_Dh(x,y)\bar{\omega}_i^{\lambda_k}(x)\bar{\omega}_i^{\lambda_k}(y)dxdy
\\
\leq\sum_{i\neq j}\int_D\int_D G(x,y)\omega_i^{\lambda_k}(x)\omega_i^{\lambda_k}(y)dxdy-\sum_{i=1}^{2}\int_D\int_Dh(x,y)\omega_i^{\lambda_k}(x)\omega_i^{\lambda_k}(y)dxdy.
\end{split}
\end{equation}
Since $\omega^\lambda_i\rightarrow\kappa_i\bm{\delta}_{x^*_i}$ and $\bar{\omega}^\lambda_i\rightarrow\kappa_i\bm{\delta}_{\bar{x}_i}$ in the sense of distribution(here $\bm{\delta}_x$ denotes the Dirac measure with unit mass concentrated at $x$), passing to the limit in \eqref{787} gives
\[ -H(\bar{x}_1,\bar{x}_2)\leq-H(x^*_1,x^*_2),\]
or equivalently
\[ H(x^*_1,x^*_2)\leq H(\bar{x}_1,\bar{x}_2).\]
Recall that $(\bar{x}_1,\bar{x}_2)$ is the unique minimum point of $H$ in $\overline{B_1}\times \overline{B_2}$, so we have
\[(x^*_1,x^*_2)=(\bar{x}_1,\bar{x}_2).\]
 This completes the proof.
\end{proof}

\subsection{$\omega^\lambda$ is a steady solution}
By combining Lemma \ref{351} and Lemma \ref{352} we know that $dist(supp(\omega^\lambda_i),\partial B_i)>0$ if $\lambda$ is large enough, which can be used to prove that $\omega^\lambda$ is a steady vortex patch.
\begin{theorem}

$\omega^\lambda$ satisfies $\eqref{997*}$ provided $\lambda$ is large enough.
\end{theorem}
\begin{proof}
For any given $\xi\in C^{\infty}_0(D)$, we define a family of $C^1$ transformations $\Phi_t(x), t\in(-\infty,+\infty)$, from $D$ to $D$ by the following dynamical system
\begin{equation}\label{400}
\begin{cases}\frac{d\Phi_t(x)}{dt}=J\nabla\xi(\Phi_t(x)),\,\,\,t\in\mathbb R, \\
\Phi_0(x)=x,
\end{cases}
\end{equation}
where $J$ denotes clockwise rotation through $\frac{\pi}{2}$ as before. Note that $\eqref{400}$ is solvable for all $t$ since $J\nabla\xi$ is a smooth vector field with compact support in $D$. It's easy to see that $J\nabla\xi$ is divergence-free, so by Liouville theorem(see \cite{MPu}, Appendix 1.1) $\Phi_t(x$) is area-preserving. Now define
\begin{equation}
\omega_{(t)}(x)=\omega^\lambda(\Phi_t(x)).
\end{equation}
Since $supp(\omega^\lambda_i)$ is away from $\partial B_i$, we have $\omega_{(t)}\in K_\lambda$ if $|t|$ is small. So $E(\omega_{(t)})$ attains its maximum at $t=0$ and thus $\frac{dE(\omega_{(t)})}{dt}=0$.

On the other hand,
\begin{equation}
\begin{split}
E(\omega_{(t)})=&\frac{1}{2}\int_D\int_DG(x,y)\omega^\lambda(\Phi_t(x))\omega^\lambda(\Phi_t(y))dxdy\\
=&\frac{1}{2}\int_D\int_DG(\Phi_{-t}(x),\Phi_{-t}(y))\omega^\lambda(x)\omega^\lambda(y)dxdy\\
=&E(\omega^\lambda)+t\int_D\omega^\lambda\partial(\psi^\lambda,\xi)+o(t),
\end{split}
\end{equation}
as $t\rightarrow 0$. So we have
\[\int_D\omega^\lambda\partial(\psi^\lambda,\xi)=0,\]
which completes the proof.
\end{proof}
\section{Energy and Stability}
In this section, we discuss the energy and stability of the vortex patch solutions obtained in Section 3.

\subsection{The notion of stability} We recall some results on the solvability of the initial value problem for the 2-D incompressible Euler equations.

The equations are as follows:
\begin{equation}\label{139}
\begin{cases}

  \partial_t\mathbf{v}+(\mathbf{v}\cdot\nabla)\mathbf{v}=-\nabla P \,\,\,\,\,\,\,\,\,\,\,\text{in $D$},\\
   \nabla\cdot\mathbf{v}=0 \,\,\,\,\,\,\,\,\,\,\,\,\,\,\,\,\,\,\,\,\,\,\,\,\,\,\,\,\,\,\,\,\,\,\,\,\,\,\,\,\,\,\,\,\,\,\text{in $D$},\\
 \mathbf{v}(x,0)=\mathbf{v}_0(x)\,\,\,\,\,\,\,\,\,\,\,\,\,\,\,\,\,\,\,\,\,\,\,\,\,\,\,\,\,\,\,\text{in $D$},
 \\ \mathbf{v}\cdot \vec{n}=0 \,\,\,\,\,\,\,\,\,\,\,\,\,\,\,\,\,\,\,\,\,\,\,\,\,\,\,\,\,\,\,\,\,\,\,\,\,\,\,\,\,\,\,\,\,\,\,\text{on $\partial D$ }.

\end{cases}
\end{equation}
Here $\mathbf{v}=(v_1,v_2)$ is the velocity field of the fluid depending on $x$ and $t$, $P$ is the pressure, $\mathbf{v}_0(x)$ is the initial velocity field and $\vec{n}$ is the outward unit normal of $\partial D$.  Similarly, by introducing the vorticity $\omega = curl\mathbf{v}$, \eqref{139} can be expressed as the following vorticity form:
\begin{equation}\label{913}
\begin{cases}
\partial_t\omega+\partial(\omega,\psi)=0\,\,\,\,\,\,\,\,\text{in $D\times(0,+\infty)$},\\
\omega(x,0)=\omega_0(x) \,\,\,\,\,\,\,\,\,\,\,\,\,\,\text{in $D$}.
\end{cases}
\end{equation}
Here as given in Section 2, $\partial(f,g)=\partial_1f\partial_2g-\partial_2f\partial_1g$ for any functions $f,g$, and $\psi(x,t)=\int_DG(x,y)\omega(y,t)dy$.
\begin{definition}
A function $\omega(x,t)\in L^{\infty}(D\times(0,+\infty))$ is called the weak solution of \eqref{913} if it satisfies
\begin{equation}\label{997}
  \int_D\omega(x,0)\xi(x,0)dx+\int_0^{+\infty}\int_D\omega(\partial_t\xi+\partial(\xi,\psi))dxdt=0
  \end{equation}
  for all $\xi\in C_0^{\infty}(D\times[0,+\infty))$, where $\psi(x,t)=\int_DG(x,y)\omega(y,t)dy$.
  \end{definition}
By Yudovich \cite{Y}, for any initial vorticity $\omega(x,0)\in L^{\infty}(D)$ there is a unique solution to $\eqref{997}$ and
 $\omega(x,t)\in L^{\infty}(D\times(0,+\infty))\cap C([0,+\infty);L^p(D)), \forall \,\,p\in [1,+\infty)$. Moreover $\omega(x,t)\in F_{\omega_0}$ for all $t\geq 0$. Here $F_{\omega}$ denotes the rearrangement class of a given function $\omega$, i.e.,
 \begin{equation}
 F_\omega\triangleq\{ v |  |\{v>a\}|=|\{\omega>a\}| ,\forall a\in \mathbb R\}.
 \end{equation}

In the sequel for simplicity we also write $\omega(x,t)$ as $\omega_t(x)$.

 Now we can give the definition of stability for steady vortex patches.
\begin{definition}
A steady vortex patch $\omega$ is called stable, if for any $\varepsilon >0$
, there exists $\delta >0$, such that for any $\omega_0\in F_\omega$, $|\omega_0-\omega|_{L^1}<\delta$, we have $|\omega_t-\omega|_{L^1}<\varepsilon$ for all $t\geq 0$. Here $\omega_t(x)=\omega(x,t)$ is the solution of $\eqref{997}$ with initial vorticity $\omega_0$.
\end{definition}
\begin{remark}
Here by stability we mean Liapunov stability, i.e., if the initial perturbation is small, then the perturbed motion is close to the steady flow for all time("closeness" is measured by $L^1$ norm).
\begin{remark}
Here for simplicity, we assume that the initial perturbation class is $F_\omega$ which in fact can be extended to a more general function class, see \cite{B3}.
\end{remark}
\end{remark}
\subsection{A stability theorem}
Now we state a stability theorem due to Burton.

Recall that for an energy conserving dynamical system in $\mathbb R^N$, a strict local minimum point of the energy must be stable, see \cite{MPu}, Chapter 3, Theorem 1.5. For the 2-D incompressible Euler flows, the energy is conserved and the vorticity moves on an isovortical surface. Based on these observations, Burton in \cite{B3} proved a stability theorem for steady vortex flows. In the case of vortex patches, his theorem can be stated as follows:
\begin{theorem}[Burton, \cite{B3}]\label{886}
Let $\omega$ be a steady vortex patch. If $\omega$ is a strict local maximizer of the kinetic energy relative to $F_\omega$,
 then $\omega$ is stable. Here $\omega$ being a strict local maximizer of the kinetic energy relative to $F_\omega$ means that there exists some $\delta_0>0$ such that for any $\bar{\omega}\in F_\omega$, $|\bar{\omega}-\omega|_{L^1}<\delta_0$ and $E(\bar{\omega})=E(\omega)$, we have $\bar{\omega}=\omega$.
\end{theorem}

To apply Theorem \ref{886}, firstly we need to prove that the steady vortex patch is a local energy maximizer in rearrangement class, and then prove that the local maximizer is in fact strict.
\subsection{Energy characterization} Now we show that $\omega^\lambda$ is a local energy maximizer in $F_{\omega^\lambda}$ if $\lambda$ is sufficiently large.

To show this, we need a non-degenerate property of $\nabla\psi^\lambda$ on $\partial\Omega_i$.

\begin{lemma}\label{576}
For sufficiently large $\lambda$, $\partial\Omega_i$ is a $C^1$ closed curve, and $\frac{\partial\psi^\lambda}{\partial\vec{n}}<0$ on $\partial\Omega_i$, where $\vec{n}$ is the outward unit normal of $\partial\Omega_i.$
\end{lemma}
\begin{proof}
The proof is exactly the same as Theorem 4.5 in \cite{T}, therefore we omit it here.

\end{proof}

\begin{remark}
One can also prove that $\partial \Omega_i$ is an infinitesimal circle and $\psi^\lambda$(after a suitable scaling)  converges to the Rankine stream function in $C^1_{loc}$ sense near each $\bar{x}_i$ as $\lambda\rightarrow +\infty$.
\end{remark}

\begin{theorem}\label{385}
For sufficiently large $\lambda$, $\omega^\lambda$ is a local energy maximizer in $F_{\omega^\lambda}$.

\end{theorem}
\begin{proof}

First we take $\lambda$ large enough such that the conclusions in Lemma \ref{576} hold. For such fixed $\lambda$, we prove this theorem by contradiction in the following.

Suppose that $\omega^\lambda$ is not a local energy maximizer $F_{\omega^\lambda}$, then we can choose a sequence $\{\omega_n\}$ and $\omega_n\in F_{\omega^\lambda}$, $0<|\omega_n-\omega^\lambda|_{L^1}<\frac{1}{n}$, and
  \begin{equation}\label{99}
E(\omega_n)> E(\omega^\lambda).
\end{equation}
By the non-degenerate property of $\nabla\psi^\lambda$ on $\partial\Omega_i$ given in Lemma \ref{576}, for large $n$ there exist $\nu_{n,1}>0$ and $\nu_{n,2}>0$ determined uniquely such that

 (1) $\partial\{\psi_n>\nu_{n,1}\}$, $\partial\{\psi_n<-\nu_{n,2}\}$ are both $C^1$ closed curves, where $\psi_n$ is the stream function of $\omega_n$;

  (2) $|\{\psi_n>\nu_{n,1}\}|=|\{\psi^\lambda>\mu^\lambda_1\}|$ and $|\{\psi_n<-\nu_{n,1}\}|=|\{\psi^\lambda<-\mu^\lambda_1\}|$;

(3) $\{\psi_n>\nu_{n,1}\} \subset B_1$ and $\{\psi_n<-\nu_{n,2}\} \subset B_2$.

 For a detailed proof of the existence of such $\nu_{n,1},\nu_{n,2}$, the reader can refer to Lemma 3.1 in \cite{CW}.

Now define $\bar{\omega}_n=\lambda (I_{\{\psi_n>\nu_{n,1}\}}-I_{\{\psi_n<-\nu_{n,2}\}})$. It is easy to verify $\bar{\omega}_n\in F_{\omega^\lambda}\cap K_\lambda(D)$, so
  \begin{equation}\label{438}
  E(\bar{\omega}_n)\leq E(\omega^\lambda).
  \end{equation}

On the other hand,
\[
\begin{split}
E(\bar{\omega}_n)-E(\omega_n)
=&\frac{1}{2}\int_D  \bar{\omega}_n\bar{\psi_n}-\frac{1}{2}\int_D \omega_n\psi_n\\
=&\int_D\psi_n(\bar{\omega}_n-\omega_n)+\frac{1}{2}\int_D(\bar{\psi}_n-\psi_n)(\bar{\omega}_n-\omega_n), \\
\end{split}
\]
where we use $\int_D\psi_n\bar{\omega}_n=\int_D\bar{\psi}_n\omega_n$ by the symmetry of Green function. Integrating by parts we have
\[\frac{1}{2}\int_D(\bar{\psi}_n-\psi_n)(\bar{\omega}_n-\omega_n)=\frac{1}{2}\int_D|\nabla(\bar{\psi}_n-\psi_n)|^2,\]
so
\[E(\bar{\omega}_n)-E(\omega_n)\geq\int_D\psi_n(\bar{\omega}_n-\omega_n).\]

Now we claim that $\int_D\psi_n(\bar{\omega}_n-\omega_n)\geq 0.$ To show this, we write $\omega_n=\lambda(I_{A_1}-I_{A_2})$ with $|A_1|=|{\psi_n>\nu_{n,1}}|$ and $|A_2|=|{\psi_n<-\nu_{n,2}}|.$  Now we calculate

\begin{equation}
\begin{split}
&\int_D\psi_n(\bar{\omega}_n-\omega_n)\\
=&\int_D\psi_n\bar{\omega}_n-\int_D\psi_n\omega_n\\
=&\lambda\int_{\{\psi_n>\nu_{n,1}\}}\psi_n-\lambda\int_{\{\psi_n<-\nu_{n,2}\}}\psi_n-\lambda\int_{A_1}\psi_n+\lambda\int_{A_2}\psi_n\\
=&\lambda(\int_{\{\psi_n>\nu_{n,1}\}}\psi_n-\int_{A_1}\psi_n)+\lambda(\int_{A_2}\psi_n-\int_{\{\psi_n<-\nu_{n,2}\}}\psi_n)\\
\geq&0.
\end{split}
\end{equation}
So we have \begin{equation}\label{819} E(\bar{\omega}_n)\geq E(\omega_n).\end{equation}
\eqref{99}, \eqref{438} and \eqref{819} together lead to a contradiction.

\end{proof}
\subsection{Uniqueness implies stability} In the last of this paper we show that uniqueness implies stability.

We state the following open problem of local uniqueness first.

$\mathbf{Open\,\, problem}$. Is there a $\lambda_0$ such that for any $\lambda>\lambda_0$ the following problem has a unique solution?

\begin{equation}
\begin{cases}
\omega=\lambda (I_{\{\psi >\mu_1\}}-I_{\{\psi <-\mu_2\}})\text{ for some $\mu_1,\mu_2\in\mathbb R$},\\
\lambda|\{\psi >\mu_1\}|=\kappa_1,\lambda|\{\psi <-\mu_2\}|=-\kappa_2,\\
\psi(x)=\int_DG(x,y)\omega(y)dy,\\
\{sgn(\kappa_i)\psi>\mu_i\}\subset B_i.
\end{cases}
\end{equation}
If the answer to the above open problem is yes, then immediately we know that $\omega^\lambda$ is the unique maximizer of $E$ in $K_\lambda(D)$. In the case $\kappa_2=0$(a single vortex patch), local uniqueness has been proved in \cite{CGPY} using the stream function method.

Combining the result of Burton, we have the following theorem.
\begin{theorem}
$\omega^\lambda$ is stable if it is the unique maximizer of $E$ in $K_\lambda(D).$
\end{theorem}
\begin{proof}
It suffices to prove that $\omega^\lambda$ is a strict local maximizer of $E$ in $F_{\omega_\lambda}$, or equivalently we need to prove that there exists a $\delta_0>0$ such that for any $\omega\in F_{\omega^\lambda} ,0<|\omega-\omega^\lambda|<\delta_0$, we have $E(\omega)< E(\omega^\lambda)$.

From the proof of Theorem \ref{385}, we know that there exists a $\delta_0>0$ such that for any $\omega\in F_{\omega^\lambda}, 0<|\omega-\omega^\lambda|<\delta_0$, there exists a $\bar{\omega}\in F_{\omega^\lambda}\cap K_\lambda(D)$ satisfying $E(\bar{\omega})\geq E(\omega)$. Since $\omega^\lambda$ is a strict local maximizer of $E$ in $F_{\omega_\lambda}$, we have $E(\bar{\omega})< E(\omega^\lambda)$. So we have $E(\omega)< E(\omega^\lambda)$,
which is the desired result.
\end{proof}

\end{document}